\documentclass[11pt, twoside, eqno]{article}
\usepackage{}
\usepackage{amssymb}
\usepackage{mathrsfs}
\usepackage{amsmath}
\usepackage{amsfonts}
\usepackage{amsthm}
\usepackage{cases}
\usepackage{colortbl,dcolumn}

\allowdisplaybreaks \numberwithin{equation}{section}

\newtheorem{theorem}{Theorem}[section]
\theoremstyle{plain}

\newtheorem{definition}{Definition}[section]

\newtheorem{lemma}{Lemma}[section]

\newtheorem{proposition}{Proposition}[section]

\numberwithin{equation}{section}
\begin{document}

\baselineskip=15pt
\renewcommand{\arraystretch}{2}
\arraycolsep=1pt

\title
{\bf \Large On the generalized porous medium equation in Fourier-Besov spaces \footnotetext{\hspace{-0.25cm}
\noindent{The research was supported by the NNSF of China under grant $\#$11601223 and $\#$ 11626213.}\\
 }}

\author
{\vspace{-0.5cm}\bf Weiliang Xiao$^1$ \quad Xuhuan Zhou$^{2*}$\\
\vspace{-0.5cm}\small\it School of Applied Mathematics,  Nanjing University of Finance and Economics, Nanjing, 210023, China;\\
\vspace{-0.5cm}\small\it E-mail:~xwltc123@163.com\\
\vspace{-0.5cm}\small\it Department of Information Technology, Nanjing Forest Police College,  Nanjing, 210023, China; \\
\vspace{-0.5cm}\small\it E-mail:~zhouxuhuan@163.com \\
}

\date{}

\maketitle

\begin{center}
\begin{minipage}{15.2cm}
\small{\bf Abstract}\quad We study a kind of generalized porous medium equation with fractional Laplacian and abstract pressure term. For a large class of equations corresponding to the form:
                $u_t+\nu \Lambda^{\beta}u=\nabla\cdot(u\nabla Pu)$,
 we get their local well-posedness in Fourier-Besov spaces for large initial data. If the initial data is small, then the solution becomes global. Furthermore, we prove a blowup criterion for the solutions. \medskip

\noindent{\bf MR (2010)}\quad 42B37, 35Q35, 35K55, 76S05

\end{minipage}
\end{center}

 \section{Introduction}
 In this paper, we study the nonlinear nonlocal equation in $\mathbb{R}^n$ of the form
      \begin{equation}\label{pme}
       \begin{cases}
        u_t+\nu \Lambda^{\alpha}u=\nabla\cdot(u\nabla Pu);   \\
        u(0,x)=u_0.
       \end{cases}
     \end{equation}
Usually, $u=u(t,x)$ ia a real-valued function, represents a density or concentration.
The dissipative coefficient $\nu>0$ corresponds to the viscid case, while $\nu=0$ 
corresponds to the inviscid case. In this paper we study the viscid case and take $\nu=1$ for simplicity. The fractional operator $\Lambda^{\alpha}$ is defined 
by Fourier transform as $(\Lambda^{\alpha}u)^{\wedge}=|\xi|^{\alpha}\hat{u}$. $P$ is an 
abstact operator.

Equation (\ref{pme}) here comes from the same proceeding with that of the fractional porous medium equation (FPME) introduced by Caffarelli and V\'azquez 
\cite{caffarelli01}.  In fact, equation (\ref{pme}) comes into being by adding the fractional dissipative term $\nu \Lambda^{\alpha}u$ to the continuity equation $u_t+\nabla\cdot(uV)=0$, where the velocity $V=-\nabla p$ and the velocity potential or pressure $p$ is related to $u$ by an abstract operator $p=Pu$. 

The absrtact form pressure term $Pu$ gives a good suitability in many cases. The 
simplest case comes from a model in groundwater in filtration \cite {bear01,vazquez01}: $u_t=\triangle u^2$, that is: $\nu=0,Pu=u$. A more general case appears in the fractional 
porous medium equation \cite{caffarelli01} when  $\nu=0$ and $Pu=\Lambda^{-2s}u,
0<s<1$. In the critical case when $s=1$, it is the mean field equation first studied by Lin
and Zhang \cite{lin01}. Some studies on the well-posedness and regularity on those equations we refer to \cite{biler02,caffarelli02,caffarelli03,serfaty01,stan01,vazquez02,zhou01} and the references therein.

In the FPME, the pressure can also be represented by Riesz potential as 
$Pu=\Lambda^{-2s}u=\mathcal{K}*u,$
with kernel 
$\mathcal{K}=c_{n,s}|y|^{2s-n}$. Replacing the kernel $\mathcal{K}$ by other functions in this form: $Pu=\mathcal{K}*u$, equation (\ref{pme} ) also appears in granular flow and biological swarming, named aggregation equation. The typical kernels are the Newton potential $|x|^{\gamma}$ and the exponent potential $-e^{-|x|}$.

One of concerned problems on this equation is the singularity of the potential $Pu$ which holds the well-posedness or leads to the blowup solution. Generally, smooth kernels at origin $x=0$ lead to the global in time solution \cite{bertozzi02}, meanwhile nonsmoooth kernels may lead to blowup phenomenon \cite{li02}. Li and Rodrigo \cite{li01,li02} studied the well-posedness and blowup criterion of equation (\ref{pme}) with the pressure $Pu=\mathcal{K}*u$, where $\mathcal{K}(x)=e^{-|x|}$ in Sobolev spaces. Wu and Zhang \cite{wu01} generalize their work to require $\nabla \mathcal{K}\in W^{1,1}$ which includes the case $\mathcal{K}(x)=e^{-|x|}$. They take advantage of the controllability in Besov spaces of the convolution $\mathcal{K}*u$ under this condition, as well as the controllability of its gradient $\nabla\mathcal{K}*u$.  

In this article we study the well-posedness and blowup criterion of equation (\ref{pme}) in Fourier-Besov spaces under an abstract pressure condition 
    \begin{align}\label{estimate11}
      \|\widehat{\triangle_k(\nabla Pu)}\|_{L^p}
      \leq C2^{k\sigma}\|\widehat{\triangle_ku}\|_{L^p}.
    \end{align}
In Fourier-Besov spaces, it is  the localization express of the norm estimate
 \begin{align}\label{estimate12}
  \|\nabla Pu\|_{F\dot{B}_{p,q}^s}\leq C\|u\|_{F\dot{B}_{p,q}^{s+\sigma}}.
 \end{align}
Corresponding to the FPME, i.e. $Pu=\Lambda^{-2s}$, we get $\sigma=1-2s$ obviously. And if $Pu=\mathcal{K}*u$, $\mathcal{K}\in W^{1,1}$ in the aggregation equation, we get $\sigma=1$ when $\mathcal{K}\in L^1$ and $\sigma=0$ when $\nabla\mathcal{K}\in L^1$.
  
 The Fourier-Besov spaces we use here come from  Konieczny and Yoneda \cite{konieczny01} when deal with the Navier-Stokes equation (NSE) with Coriolis force. Besides, Fourier-Besov spaces have been widely used to study the well-posedness, singularity, self-similar solution, etc. of Fluid Dynamics
in various of forms. For instance, the early pseudomeasure spaces $PM^{\alpha}$ in which Cannone and Karch studied the smooth and singular properties of Navier-Stokes equations \cite{cannone01}. The Lei-Lin spaces $\mathcal{X}^{\sigma}$ deal with global solutions to the NSE \cite{lei01} and to the quasi-geostrophic equations (QGE) \cite{benameur01}. The Fourier-Herz spaces $\mathcal{B}_q^{\sigma}$ in the Keller-Segel system \cite{iwabuchi01}, in the NSE with Coriolis force \cite{iwabuchi02} and in the magneto-hydrodynamic equations (MHD) \cite{liu01}. 

In the case of Besov spaces, we gain some well-posedness and blow-up results of equation (\ref{pme}) under an corresponding condition to (\ref{estimate12}) in \cite{zhou02}. Due to the difficulty in deal with the nonlinear term $\nabla\cdot(u\nabla Pu)$, in that case we need a little strict initial condition: $u_0\in\dot{B}_{p,1}^{n/p+\sigma-\beta}\cap\dot{B}_{p,1}^{n/p+\sigma-\beta+1}$. However, in this paper, we find that Fourier-Besov spaces are powerful in deal with the nonlinear term, by a very different method used in \cite{zhou02}, we prove the following theorems:

\begin{theorem}
Let $p,q\in[1,\infty]$, $\max\{1,\sigma+1\}<\alpha<n(1-\frac 1p)+\sigma+2$. Then for any $u_0\in\dot{FB}^{n(1-\frac 1p)-\alpha+\sigma+1}_{p,q}$,
equation (1.1) admits a unique mild solution $u$ and
 $$u\in C\big([0,T);\dot{FB}^{n(1-\frac 1p)-\alpha+\sigma+1}_{p,q}\big)\cap \tilde{L}^{1}([0,T);\dot{FB}^{n(1-\frac 1p)+\sigma+1}_{p,q}).$$
 Moreover, there exists a constant $C_0=C_0(\alpha,p,q)$ such that for $u_0$ satisfying $\|u_0\|_{\dot{FB}^{n(1-\frac 1p)-\alpha+\sigma+1}_{p,q}}\leq C_0$,
 the solution $u$ is global, and
 \begin{align*}
\|u\|_{L^{\infty}_T(\dot{FB}^{n(1-\frac 1p)-\alpha+\sigma+1}_{p,q})}+\|u\|_{\tilde{L}^{1}_T(\dot{FB}^{n(1-\frac 1p)+\sigma+1}_{p,q})}
\leq 2C\|u_0\|_{\dot{FB}_{p,q}^{n(1-\frac 1p)-\alpha+\sigma+1}}.
\end{align*}
\end{theorem}

   \begin{theorem}
       Let $T^*$ denote the maximal time of existence of $u$ in
       $L^{\infty}_T(\dot{FB}^{\beta}_{p,q})\cap\tilde{L}^{1}_T(\dot{FB}^{\beta+\alpha}_{p,q})$. Here $\beta=n(1-\frac 1p)-\alpha+\sigma+1$.
       If $T^*<\infty$, then
        \begin{align*}
           \|u\|_{\tilde{L}^1([0,T^*);\dot{FB}_{p,q}^{\beta+\alpha})}=\infty.
        \end{align*}
   \end{theorem}

 \section{Preliminaries}
 Let us introduce some basic knowledge on Littlewood-Paley theory and Fourier-Besov spaces.
 Let $\varphi\in C_c^{\infty}(\mathbb{R}^n)$ be a radial positive function such that
    $$supp\ \varphi\subset\{\xi\in \mathbb{R}^n:\frac34\leq|\xi|\leq\frac83\},\quad
      \sum_{j\in Z}\varphi(2^{-j}\xi)=1\ \text{for any}\ \xi\neq0.$$
 Define the frequency localization operators as follows:
    \begin{align*}
          \triangle_ju=\varphi_j(D)u;\ \  S_ju=\psi_j(D)u,
    \end{align*}
here $\varphi_{_j}(\xi)=\varphi(2^{-j}\xi)$ and $\psi_j=\sum_{k\leq j-1}\varphi_{_j}$.

By Bony's decomposition we can split the product $uv$ into three parts:
\begin{align*}
          uv=T_uv+T_vu+R(u,v),
    \end{align*}
with
\begin{align*}
         T_uv=\sum_jS_{j-1}u\Delta_jv , \ \ R(u,v)=\sum_j\Delta_ju\tilde{\Delta}_jv,\ \ \tilde{\Delta}_jv=\Delta_{j-1}v+\Delta_jv+\Delta_{j+1}v.
    \end{align*}
Let us now define the Fourier-Besov space as follows.
\begin{definition}\label{definition21}
     For ~$\beta\in \mathbb{R},p,q\in[1,\infty]$, we define the Fourier-Besov space $\dot{FB}_{p,q}^\beta$ as
        \begin{align*}
          \dot{FB}_{p,q}^\beta=\{f\in \mathscr{S}'/\mathbb{P}:\|f\|_{\dot{FB}_{p,q}^\beta}
                         =\Big(\sum_{j\in\mathbb{Z}}2^{j\beta q}\|\varphi_j\hat{f}\|_{L^p}^q\Big)^{1/q}<\infty\}.
        \end{align*}
     Here the norm changes normally when $p=\infty$ or $q=\infty$, and $\mathbb{P}$ is the set of all polynomials.
\end{definition}
\begin{definition}\label{definition22}
 In this paper, we need two kinds of mixed time-space norm defined as follows: For $s\in \mathbb{R},1\leq p,q\leq\infty,I=[0,T),T\in(0,\infty]$,
 and let X be a Banach space with norm $\|\cdot\|_X$,
    \begin{align*}
     \|f(t,x)\|_{L^r(I;X)}&:=\big(\int_I\|f(\tau,\cdot)\|^r_{X}d\tau\big)^{\frac1r},  \\
     \|f(t,x)\|_{\tilde{L}^r(I;\dot{FB}_{p,q}^\beta)}&:=\big(\sum_{j\in\mathbb{Z}}2^{j\beta q}\|\varphi_j\hat{f}\|^q_{L^r(I;L^p)}\big)^{\frac1q}.
    \end{align*}
 By Minkowski' inequality, there holds
    \begin{align}\label{estimate21}
     L^r(I;F\dot{B}_{p,q}^s)\hookrightarrow \tilde{L}(I;F\dot{B}_{p,q}^s),\ \text{if}\ r\leq q
     \quad\text{and}\quad \tilde{L}^r(I;F\dot{B}_{p,q}^s)\hookrightarrow L^r(I;F\dot{B}_{p,q}^s),\ \text{if}\ r\geq q.
    \end{align}
\end{definition}
\begin{lemma}\label{absolution02}
        Let $X$ be a Banach space with norm $\|\cdot\|_X$ and $B:X\times X\mapsto X$ be a bounded
        bilinear operator satisfying
            $$\|B(u,v)\|_X\leq\eta\|u\|_X\|v\|_X,$$
        for all $u,v\in X$ and a constant $\eta>0$. Then for any fixed $y\in X$ satisfying $\|y\|_X<\epsilon<\frac{1}{4\eta}$,
        the equation $x:=y+B(x,x)$ has a solution $\bar{x}$ in $X$ such that $\|\bar{x}\|_X\leq2\|y\|_X$. Also, the solution
        is unique in $\bar{B}(0,2\epsilon)$. Moreover, the solution depends continuously on $y$ in the sense: if $\|y'\|_X<\epsilon,x'=y'+B(x',x'),\|x'\|_X<2\epsilon$, then
            $$\|\bar{x}-x'\|_X\leq\frac{1}{1-4\epsilon\eta}\|y-y'\|_X.$$
    \end{lemma}
\begin{lemma}\label{interpolation01}\cite{xiao01} If $\frac{1}{r}=\frac{1-\theta}{r_1}+\frac{1-\theta}{r_1}$, then
    \[
      \|u\|_{\tilde{L}^r_T(\dot{FB}_{p,q}^{\beta+\theta \alpha})}
      \leq   \|u\|^{1-\theta}_{\tilde{L}^{r_1}_T(\dot{FB}_{p,q}^{\beta})}
             \|u\|^{\theta}_{\tilde{L}^{r_2}_T(\dot{FB}_{p,q}^{\beta+\alpha})}.
    \]
\end{lemma}
Now we prove a priori estimate which will be used in our proof. Consider the following dissipative equation:
    \begin{align}\label{equation22}
     \partial_t u+\Lambda^\alpha u=f(t,x),\quad u(0,x)=u_0(x),\quad t>0,x\in \mathbb{R}^n.
    \end{align}
  \begin{lemma}\label{transport}
    Let $0<T\leq\infty,\beta\in \mathbb{R}$ and $1\leq r\leq \infty$. Assume $u_0\in\dot{FB}_{p,q}^\beta$, $f\in \tilde{L}^1(I;\dot{FB}_{p,q}^\beta)$. Then
    the solution $u(t,x)$ to (2.2) satisfies
    \begin{align}
    \|u\|_{\tilde{L}^r(I;\dot{FB}_{p,q}^{\beta+\frac{\alpha}{r}})}
    \leq C(\|u_0\|_{\dot{FB}_{p,q}^\beta}+\|f\|_{\tilde{L}^1(I;\dot{FB}_{p,q}^\beta)})
    \end{align}
  \end{lemma}
\begin{proof}
Consider the integral form of (\ref{equation22})
   \[
     u(t,x)=e^{-t\Lambda^{\alpha}}u_0+\int_0^t e^{-(t-\tau)\Lambda^{\alpha}}f(\tau,x){\rm d}\tau:=Lu_0+Gf.
   \]
For the linear part,
   \begin{align*}
    \|\varphi_j \mathcal{F}(Lu_0)\|_{L^p}
    \leq e^{-t2^{j\alpha}(3/4)^{\alpha}}\|\varphi_j \widehat{u_0}\|_{L^p}.
   \end{align*}
Hence
   \begin{align*}
     \|Lu_0\|_{\tilde{L}_T^r(\dot{FB}_{p,q}^{\beta+\frac{\alpha}{r}})}
     \leq \big\|2^{j(\beta+\frac{\alpha}{r})}\|e^{-t2^{j\alpha}(3/4)^{\alpha}}\|_{L_T^r}\|\varphi_j \widehat{u_0}\|_{L^p}\big\|_{l^q}
     \leq C\|u_0\|_{\dot{FB}_{p,q}^{\beta}}.
   \end{align*}
On the other hand, for the integral part,
   \begin{align*}
     \|\varphi_j\mathcal{F}(Gf)\|_{L^p}
     \leq \int_0^t e^{-(t-\tau)(\frac34 2^j)^{\alpha}}\|\varphi_j\hat{f}\|_{L^p}{\rm d}\tau
   \end{align*}
Taking $L^r$-norm with respect to time, by Minkowskii's inequality
      \begin{align*}
     \|\varphi_j\mathcal{F}(Gf)\|_{L_T^r(L^p)}
     \leq C2^{-\frac{j\alpha}{r}}\|\varphi_j\hat{f}\|_{L_T^1(L^p)}.
   \end{align*}
Hence
     \begin{align*}
      \|Gf\|_{\tilde{L}_T^r(\dot{FB}_{p,q}^{\beta+\frac{\alpha}{r}})}
      = \big\|2^{j(\beta+\frac{\alpha}{r})}\|\varphi_j\mathcal{F}(Gf)\|_{L_T^r(L^p)}\big\|_{l^q}
      \leq C\|f\|_{\tilde{L}_T^1(\dot{FB}_{p,q}^{\beta})}.
     \end{align*}
Combine the above estimates, we obtain our desire inequality.
\end{proof}

\section{Local and global well-posedness}
  In this section we prove our main Theorem.
 First we know that the integral form of $u$ is as follows
  \begin{align}\label{equation31}
  \begin{split}
        u&=e^{-t\Lambda^\alpha}u_0+\int_0^te^{-(t-\tau)\Lambda^\alpha}\nabla\cdot(u(\tau)\nabla Pu(\tau)){\rm d}\tau\\
        &:=S(t)u_0+H(u,u).
    \end{split}
  \end{align}
  Here $S(t)u_0=\mathcal{F}^{-1}(e^{-t|\xi|^\alpha}\hat{u}_0)$, and 
  $H(u.v)=\int_0^te^{-(t-\tau)\Lambda^\alpha}\nabla\cdot(u(\tau)\nabla Pv(\tau)){\rm d}\tau$.
Now we get the following estimate
\begin{proposition}\label{propositon31}
       Let $\gamma,p,q\geq 1$, $\epsilon>\max\{0,-\sigma\}$, $\beta>0$, $\frac{1}{\gamma}=\frac{1}{\gamma_1}+\frac{1}{\gamma_2}$,
       there holds
       \begin{align*}
\|u\partial_iPv\|_{\tilde{L}^{\gamma}_t(\dot{FB}^{\beta}_{p,q})}
&\leq C\|u\|_{\tilde{L}^{\gamma_1}_t(\dot{FB}^{n(1-\frac 1p)-\epsilon}_{p,q})}\|v\|_{\tilde{L}^{\gamma_2}_t(\dot{FB}^{\beta+\sigma+\epsilon}_{p,q})}\\
&+C\|v\|_{\tilde{L}^{\gamma_1}_t(\dot{FB}^{n(1-\frac 1p)-\epsilon}_{p,q})}\|u\|_{\tilde{L}^{\gamma_2}_t(\dot{FB}^{\beta+\sigma+\epsilon}_{p,q})}.
\end{align*}
   \end{proposition}

\begin{proof}
By the Bony's decomposition, it is easy to get that
\begin{align}
\Delta_j(u\partial_iPv)&=\sum_{|k-j|\leq 4}\Delta_j(S_{k-1}u\Delta_k(\partial_iPv))\\&
+\sum_{|k-j|\leq 4}\Delta_j(S_{k-1}(\partial_iPv)\Delta_ku)+\sum_{k\geq j-2}\Delta_j({\Delta_ku\tilde{\Delta}_k(\partial_iPv)})
\end{align}
We can get the following estimates:
\begin{align*}
&2^{j\beta}\|\sum_{|k-j|\leq 4}\varphi_j\mathcal{F}(S_{k-1}u\Delta_k(\partial_iPv))\|_{L^{\gamma}_t(L^p)}\\
&\leq 2^{j\beta}\sum_{|k-j|\leq 4}\|\mathcal{F}(S_{k-1}u)\ast\varphi_k\mathcal{F}(\partial_iPv)\|_{L^{\gamma}_t(L^p)}\\
&\leq C2^{j\beta}\sum_{|k-j|\leq 4}\big\|\sum_{l\leq k-2}2^{ln(1-\frac 1p)}\|\varphi_l\hat{u}\|_{L^p}2^{k\sigma}
 \|\varphi_k\hat{v}\|_{L^p}\big\|_{L_t^{\gamma}}\\
&\leq C2^{j\beta}\sum_{|k-j|\leq 4}\big\|\sum_{l\leq k-2}2^{(l-k)\epsilon}2^{nl(1-\frac 1p)-\epsilon l}
 \|\varphi_l\hat{u}\|_{L^p}2^{k(\sigma+\epsilon)}\|\varphi_k\hat{v}\|_{L^p}\big\|_{L_t^{\gamma}}\\
&\leq C \|u\|_{\tilde{L}^{\gamma_1}_t(\dot{FB}^{n(1-\frac 1p)-\epsilon}_{p,q})}2^{j(\beta+\sigma+\epsilon)}\|\varphi_j\hat{v}\|_{L^{\gamma_2}_t(L^p)}.
\end{align*}
Similarly, we can estimate
\begin{align*}
&2^{j\beta}\|\sum_{|k-j|\leq 4}\varphi_j\mathcal{F}(S_{k-1}(\partial_iPv)\Delta_ku)\|_{L^{\gamma}_t(L^p)}\\
&\leq C\|v\|_{\tilde{L}^{\gamma_1}_t(\dot{FB}^{n(1-\frac 1p)-\epsilon}_{p,q})}2^{j(\beta+\sigma+\epsilon)}\|\varphi_k\hat{u}\|_{L^{\gamma_2}_t(L^p)}.
\end{align*}
 And when $\beta>0$, we can also get that for some $\|a_j\|_{l^q}=1$,
\begin{align*}
&2^{j\beta}\|\mathcal{F}(\sum_{k\geq j-2}\Delta_j(\Delta_k(\partial_iPv)\tilde{\Delta}_ku)\|_{L^{\gamma}_t(L^p)}\\
&\leq 2^{j\beta}\|\sum_{k\geq j-2}\varphi_j\mathcal{F}(\Delta_k(\partial_iPv))\ast(\tilde{\varphi}_k\hat{u})\|_{L^{\gamma}_t(L^p)}\\
&\leq C2^{j\beta}\big\|\sum_{k\geq j-2}2^{kn(1-\frac 1p)+k\sigma}\|\varphi_k\hat{v}\|_{L^p}
 \|\tilde{\varphi}_k\hat{u}\|_{L^p}\big\|_{L^{\gamma}_t}\\
&\leq C\sum_{k\geq j-2}2^{(j-k)\beta}\big\|2^{kn(1-\frac 1p)-\epsilon k}\|\varphi_k\hat{v}\|_{L^p}
 2^{k(\beta+\sigma+\epsilon)}\|\tilde{\varphi}_k\hat{u}\|_{L^p}\big\|_{L^{\gamma}_t}    \\
&\leq C\sum_{k\geq j-2}2^{(j-k)\beta}\|v\|_{\tilde{L}^{\gamma_1}_t(\dot{FB}^{n(1-\frac 1p)-\epsilon}_{p,q})}
 2^{k(\beta+\sigma+\epsilon)}\|\tilde{\varphi}_ku\|_{\tilde{L}^{\gamma_2}_t(L^p)}\\
&\leq C a_j\|v\|_{\tilde{L}^{\gamma_1}_t(\dot{FB}^{n(1-\frac 1p)-\epsilon}_{p,q})}\|u\|_{\tilde{L}^{\gamma_2}_t(\dot{FB}^{\beta+\sigma+\epsilon}_{p,q})}.
\end{align*}
Combine the above estimate, we obtain the following inequality
\begin{align*}
 \|u\partial_iPv\|_{\tilde{L}^{\gamma}_t(\dot{FB}^{\beta}_{p,q})}
&\leq C\|u\|_{\tilde{L}^{\gamma_1}_t(\dot{FB}^{n(1-\frac 1p)-\epsilon}_{p,q})}\|v\|_{\tilde{L}^{\gamma_2}_t(\dot{FB}^{\beta+\sigma+\epsilon}_{p,q})}\\
&+C\|v\|_{\tilde{L}^{\gamma_1}_t(\dot{FB}^{n(1-\frac 1p)-\epsilon}_{p,q})}\|u\|_{\tilde{L}^{\gamma_2}_t(\dot{FB}^{\beta+\sigma+\epsilon}_{p,q})}.
\end{align*}
\end{proof}
\begin{proposition} \label{propositon32}
       Let $p,q\geq 1$, $\epsilon>\max\{0,-\sigma\}$, $\beta>-1$. If $u,v\in \tilde{L}^{2}_T(\dot{FB}^{n(1-\frac 1p)-\epsilon}_{p,q})\cap\tilde{L}^{2}_T(\dot{FB}^{\beta+\sigma+\epsilon+1}_{p,q})$, we have
\begin{align*}
&\|H(u,v)\|_{\tilde{L}^{\infty}_T(\dot{FB}^{\beta}_{p,q})}+\|H(u,v)\|_{\tilde{L}^{1}_T(\dot{FB}^{\beta+\alpha}_{p,q})}  \\
&\leq C\|u\|_{\tilde{L}^{2}_T(\dot{FB}^{n(1-\frac 1p)-\epsilon}_{p,q})}\|v\|_{\tilde{L}^{2}_T(\dot{FB}^{\beta+\sigma+\epsilon+1}_{p,q})}
  +C\|v\|_{\tilde{L}^{2}_T(\dot{FB}^{n(1-\frac 1p)-\epsilon}_{p,q})}\|u\|_{\tilde{L}^{2}_T(\dot{FB}^{\beta+\sigma+\epsilon+1}_{p,q})}.
\end{align*}
   \end{proposition}
   \begin{proof}
   We note that $H(u,v)$ is a solution to equation (\ref{equation22}) with $u_0=0,f=\nabla \cdot (u\nabla Pv)$.
So by Lemma \ref{transport} there holds
\begin{align*}
\|H(u,v)\|_{\tilde{L}^{\infty}_T(\dot{FB}^{\beta}_{p,q})}+\|H(u,v)\|_{\tilde{L}^{1}_T(\dot{FB}^{\beta+\alpha}_{p,q})}
\leq C\|\nabla \cdot (u\nabla Pv)\|_{\tilde{L}_T^1(\dot{FB}_{p,q}^{\beta})}.
\end{align*}
Then Proposition \ref{propositon31} gives the estimate.
\end{proof}
\begin{theorem}\label{theorem31}
Let $p,q\in[1,\infty]$, $2\max\{1,\sigma+1\}<\alpha<n(1-\frac 1p)+\sigma+2$. Then for any $u_0\in\dot{FB}^{n(1-\frac 1p)-\alpha+\sigma+1}_{p,q}$,
equation (1.1) admits a unique mild solution $u$ and
 $$u\in C\big([0,T);\dot{FB}^{n(1-\frac 1p)-\alpha+\sigma+1}_{p,q}\big)\cap \tilde{L}^{1}([0,T);\dot{FB}^{n(1-\frac 1p)+\sigma+1}_{p,q}).$$
 Moreover, there exists a constant $C_0=C_0(\alpha,p,q)$ such that for $u_0$ satisfying $\|u_0\|_{\dot{FB}^{n(1-\frac 1p)-\alpha+\sigma+1}_{p,q}}\leq C_0$,
 the solution $u$ is global, and
 \begin{align*}
\|u\|_{L^{\infty}_T(\dot{FB}^{n(1-\frac 1p)-\alpha+\sigma+1}_{p,q})}+\|u\|_{\tilde{L}^{1}_T(\dot{FB}^{n(1-\frac 1p)+\sigma+1}_{p,q})}
\leq 2C\|u_0\|_{\dot{FB}_{p,q}^{n(1-\frac 1p)-\alpha+\sigma+1}}.
\end{align*}
\end{theorem}
\begin{proof}
First suppose $t\in[0,T]$, $T$ fixed.
Let $\epsilon=\frac{\alpha}{2}-\sigma-1$, $\beta=n(1-\frac 1p)-\alpha+\sigma+1$, by the above proposition
\begin{align*}
&\|H(u,v)\|_{L^{\infty}_T(\dot{FB}^{n(1-\frac 1p)-\alpha+\sigma+1}_{p,q})}+\|H(u,v)\|_{\tilde{L}^{1}_T(\dot{FB}^{n(1-\frac 1p)+\sigma+1}_{p,q})}
\\&\leq C\|u\|_{\tilde{L}^{2}_T(\dot{FB}^{n(1-\frac 1p)-\frac{\alpha}{2}+\sigma+1}_{p,q})}
         \|v\|_{\tilde{L}^{2}_T(\dot{FB}^{n(1-\frac 1p)-\frac{\alpha}{2}+\sigma+1}_{p,q})}.
\end{align*}
Define $X=\tilde{L}^{2}_T(\dot{FB}^{n(1-\frac 1p)-\frac{\alpha}{2}+\sigma+1}_{p,q})$, by Lemma \ref{interpolation01}
\begin{align*}
&\|H(u,v)\|_X\leq C\|u\|_X\|v\|_X.
\end{align*}
By Lemma \ref{absolution02} we know that if $\|e^{-t\Lambda^\alpha}u_0\|_X<\frac {1}{4C}$, then (\ref{equation31}) has a unique solution in $B(0,\frac {1}{2C})$.
Here $B(0,\frac {1}{2C}):=\{x\in X: \|x\|_X\leq \frac {1}{2C}\}$.

To conclude $\|e^{-t\Lambda^\alpha}u_0\|_X<\frac {1}{4C}$, we first note that $e^{-t\Lambda^\alpha}u_0$ is the solution to (2.2) with $f=0, u_0=u_0$,
by Lemma \ref{transport}, we can obtain
\begin{align}\label{estimate01}
\|e^{-t\Lambda^\alpha}u_0\|_{X}\leq C\|u_0\|_{\dot{FB}_{p,q}^{n(1-\frac 1p)-\alpha+\sigma+1}}.
\end{align}
Hence if $\|u_0\|_{\dot{FB}_{p,q}^{n(1-\frac 1p)-\alpha+\sigma+1}}\leq \frac{1}{4C^2},$ (3.1) has a unique global solution in $X$.
Moreover, $\|u\|_X\leq 2C\|u_0\|_{\dot{FB}_{p,q}^{n(1-\frac 1p)-\alpha+\sigma+1}}$.

On the other hand, denote
$u_0=\mathcal{F}^{-1}\chi_{\{|\xi|\leq \lambda\}}\widehat{u_0}+\mathcal{F}^{-1}\chi_{\{|\xi|>\lambda\}}\widehat{u_0}:=u_0^1+u_0^2$,
where $\lambda=\lambda(u_0)>0$ is a real number determined later. Since $u_0^2$ converges to 0 in $\dot{FB}_{p,q}^{n(1-\frac 1p)-\alpha+\sigma+1}$
as $\lambda\rightarrow+\infty$. By (\ref{estimate01}) there exists $\lambda$ large enough such that
 \[
  \|e^{-t\Lambda^\alpha}u_0^2\|_X\leq \frac{1}{8C}.
 \]
For $u_0^1$,
 \begin{align*}
   \|e^{-t\Lambda^\alpha}u_0^1\|_X
    &=\big\|2^{j(n(1-\frac 1p)-\frac{\alpha}{2}+\sigma+1)}\|\varphi_je^{-t|\xi|^\alpha}\chi_{\{|\xi|\leq \lambda\}}\widehat{u_0}\|_{L_T^2(L^p)}\big\|_{l^q}   \\
    & \leq \big\|2^{j(n(1-\frac 1p)-\frac{\alpha}{2}+\sigma+1)}
      \|\sup_{|\xi|\leq \lambda}e^{-t|\xi|^\alpha}|\xi|^{\frac{\alpha}{2}}\|_{L_T^2}
      \ \|\varphi_j|\xi|^{-\frac{\alpha}{2}}\hat{u_0}\|_{L^p}\big\|_{l^q}   \\
    &\leq C \lambda^{\frac{\alpha}{2}} T^{\frac12}\|u_0\|_{\dot{FB}_{p,q}^{n(1-\frac 1p)-\alpha+\sigma+1}}.
 \end{align*}
Thus for arbitrary $u_0$ in $\dot{FB}_{p,q}^{n(1-\frac 1p)-\alpha+\sigma+1}$, (3.1) has a unique local solution in $X$ on $[0,T)$ where
    \[
      T\leq\Big(\frac{1}{8C^2\lambda^{\frac{\alpha}{2}}\|u_0\|_{\dot{FB}_{p,q}^{n(1-\frac 1p)-\alpha+\sigma+1}}}\Big)^2.
    \]
 The continuity with respect to time is standard.
\end{proof}
Next we give a blowup criterion as following:
   \begin{theorem}
       Let $T^*$ denote the maximal time of existence of $u$ in
       $L^{\infty}_T(\dot{FB}^{\beta}_{p,q})\cap\tilde{L}^{1}_T(\dot{FB}^{\beta+\alpha}_{p,q})$. Here $\beta=n(1-\frac 1p)-\alpha+\sigma+1$.
       If $T^*<\infty$, then
        \begin{align}
           \|u\|_{\tilde{L}^1([0,T^*);\dot{FB}_{p,q}^{\beta+\alpha})}=\infty.
        \end{align}
   \end{theorem}
 \begin{proof}
  Supposing $\|u\|_{\tilde{L}^1([0,T^*);\dot{FB}_{p,q}^{\beta+\alpha})}dt<\infty$, then we can find $0<T_0<T^*$ satisfying
  \begin{align*}
       \|u\|_{\tilde{L}_1([T_0,T^*);\dot{FB}_{p,q}^{\beta+\alpha})}<\frac 12.
    \end{align*}
  For $t\in [T_0,T^*]$, $s\in[T_0,t]$, by Lemma \ref{transport}
  \begin{align*}
  \|u(s)\|_{\dot{FB}_{p,q}^{\beta}}+\|u\|_{\tilde{L}^1([T_0,s);\dot{FB}_{p,q}^{\beta+\alpha})}&\leq
   \|u(T_0)\|_{\dot{FB}_{p,q}^{\beta}}+\|u\|_{L^\infty([T_0,s);\dot{FB}_{p,q}^{\beta})}\|u\|_{\tilde{L}^1([T_0,s);\dot{FB}_{p,q}^{\beta+\alpha})}\\
   &\leq\|u(T_0)\|_{\dot{FB}_{p,q}^{\beta}}+\frac 12\|u\|_{L^\infty([T_0,s);\dot{FB}_{p,q}^{\beta})}.
  \end{align*}
  So,
    \begin{align*}
  \sup_{T_0\leq s\leq t}\|u(s)\|_{\dot{FB}_{p,q}^{\beta}}\leq\|u(T_0)\|_{\dot{FB}_{p,q}^{\beta}}+\frac 12\|u\|_{L^\infty([T_0,t);\dot{FB}_{p,q}^{\beta})}.
  \end{align*}
  Put
  \begin{align*}
  M=\max(2\|u(T_0)\|_{\dot{FB}_{p,q}^{\beta}},\max_{t\in[0,T_0]}\|u\|_{\dot{FB}_{p,q}^{\beta}}),
  \end{align*}
  we can get
   \begin{align*}
  \|u(t)\|_{\dot{FB}_{p,q}^{\beta}}\leq M,\quad   \forall t \in[0,T^*].
  \end{align*}
  On the other side,
  \begin{align*}
  u(t_2)-u(t_1)&=e^{-t_2|D|^\alpha}u_0-e^{-t_1|D|^\alpha}u_0+\int_0^{t_2}e^{-(t_2-\tau)|D|^\alpha}\nabla\cdot(uPu)(\tau){\rm d}\tau\\
  &-\int_0^{t_1}e^{-(t_1-\tau)|D|^\alpha}\nabla\cdot(uPu)(\tau){\rm d}\tau\\
  &=[e^{-t_2|D|^\alpha}u_0-e^{-t_1|D|^\alpha}u_0]+[\int_0^{t_1}e^{-(t_1-\tau)|D|^\alpha}(e^{-(t_2-t_1)|D|^\alpha}-1)\nabla\cdot(uPu)(\tau){\rm d}\tau]\\
  &+[\int_{t_1}^{t_2}e^{-(t_2-\tau)|D|^\alpha}\nabla\cdot(uPu)(\tau){\rm d}\tau]\\
  &:=L_1+L_2+L_3.
 \end{align*}
   \begin{align*}
   \|L_1\|_{\dot{FB}_{p,q}^{\beta}}&=\big\|2^{j\beta}\|\varphi_j(e^{-t_2|\xi|^\alpha}-e^{-t_1|\xi|^\alpha})\hat{u}_0\|_{L^p}\big\|_{l^q}\\
   &\leq\big\|2^{j\beta}\|\varphi_j(e^{-(t_2-t_1)|\xi|^\alpha}-1)\hat{u}_0\|_{L^p}\big\|_{l^q}\\
   &\leq \|u_0\|_{\dot{FB}_{p,q}^{\beta}}.
     \end{align*}
  \begin{align*}
   \|L_2\|_{\dot{FB}_{p,q}^{\beta}}&\leq\big\|2^{j\beta}\int_0^{t_1}
   \|\varphi_je^{-(t_2-\tau)|\xi|^\alpha}(1-e^{-(t_2-t_1)|\xi|^\alpha})\mathcal{F}(\nabla\cdot(uPu)(\tau))\|_{L^p}{\rm d}\tau\big\|_{l^q}\\
   &\leq\big\|2^{j(\beta+1)}\int_0^{t_1}\|\varphi_j(e^{-(t_2-t_1)|\xi|^\alpha}-1)\mathcal{F}(uPu)(\tau)\|_{L^p}{\rm d}\tau\big\|_{l^q}.
     \end{align*}
     \begin{align*}
    \|L_3\|_{\dot{FB}_{p,q}^{\beta}}&\leq\big\|2^{j\beta}\int_{t_1}^{t_2}
    \|\varphi_je^{-(t_2-\tau)|\xi|^\alpha}\mathcal{F}(\nabla\cdot(uPu)(\tau))\|_{L^p}{\rm d}\tau\big\|_{l^q}\\
   &\leq\big\|2^{j(\beta+1)}\int_{t_1}^{t_2}\|\varphi_j\mathcal{F}(uPu)(\tau)\|_{L^p}{\rm d}\tau\big\|_{l^q}.
     \end{align*}
     By the dominated convergence theorem, we can get
     \begin{align*}
    \limsup_{t_1,t_2\nearrow T^*,t_1<t_2}\|u(t_2)-u(t_1)\|_{\dot{FB}_{p,q}^{\beta}}=0.
     \end{align*}
     Then there is an element $u^*$ of $\dot{FB}_{p,q}^{\beta}$ such that
     \begin{align*}
    \lim_{t\nearrow T^*}\|u(t)\|_{\dot{FB}_{p,q}^{\beta}}=u^*.
     \end{align*}
Now set $u(T^*)=u^*$ and consider the equation starting by $u^*$, by the well-posedness we obtain a solution
  existing on a larger time interval than $[0,T^*)$, which is a contradiction.
 \end{proof}

\section{Improvement of the index}
The index in Theorem \ref{theorem31} is not a nature one since $\alpha>2$ is a very strong condition. However,
the method used here can in fact gain some better index range by a slight modification. The key point is that
we can seek the solution firstly in the space
     $X:={\tilde{L}^{r}_T(\dot{FB}^{\beta+\frac{\alpha}{r}}_{p,q})} \cap {\tilde{L}^{r'}_T(\dot{FB}^{\beta+\frac{\alpha}{r'}}_{p,q})}$
instead of $\tilde{L}^{2}_T(\dot{FB}^{\beta+\frac{\alpha}{2}}_{p,q})$, with $\beta=n(1-\frac{1}{p})-\alpha+\sigma+1$. Since the proof are very similar, we list the key
steps.

\textbf{Step 1:} Taking respectively $r$ and $r'$ in Lemma \ref{transport}, and using the proof of Proposition  \ref{propositon32}, we get
for $\epsilon>\max\{0,-\sigma\}$, $\beta>-1$,
    \begin{align*}
     &\|H(u,v)\|_{\tilde{L}^{r}_T(\dot{FB}^{\beta+\frac{\alpha}{r}}_{p,q})}+\|H(u,v)\|_{\tilde{L}^{r'}_T(\dot{FB}^{\beta+\frac{\alpha}{r'}}_{p,q})}  \\
     &\leq C\|u\|_{\tilde{L}^{r}_T(\dot{FB}^{n(1-\frac 1p)-\epsilon}_{p,q})}\|v\|_{\tilde{L}^{r'}_T(\dot{FB}^{\beta+\sigma+\epsilon+1}_{p,q})}
      +C\|v\|_{\tilde{L}^{r}_T(\dot{FB}^{n(1-\frac 1p)-\epsilon}_{p,q})}\|u\|_{\tilde{L}^{r'}_T(\dot{FB}^{\beta+\sigma+\epsilon+1}_{p,q})}.
\end{align*}
Set $\epsilon=\frac{\alpha}{r'}-\sigma-1$, $\beta=n(1-\frac{1}{p})-\alpha+\sigma+1$, we then gain the important bilinear estimate
$H(u,v)\leq C\|u\|_{X} \|v\|_{X}$ under the condition: 
$$r'\max\{1,\sigma+1\}<\alpha<n(1-\frac{1}{p})+\sigma+2.$$
Thus by Lemma \ref{absolution02}
we know that if $\|e^{-t\Lambda^{\alpha}}u_0\|_X<\frac{1}{4C}$, then equation (\ref{equation31}) admits a unique solution in $X$.

\textbf{Step 2:} Now we need to derive $\|e^{-t\Lambda^{\alpha}}u_0\|_X<\frac{1}{4C}$. Since $e^{-t\Lambda^\alpha}u_0$ is the solution to (\ref{equation22})
with $f=0, u_0=u_0$, by Lemma \ref{transport} we gain the global solution in $X$ for small initial data. On the other hand, we can also obtain the local
solution on $[0,T)$ in $X$ by the same method in Theorem \ref{theorem31} for arbitrary initial data, where
    \[
      T\leq\min\Big\{\Big(\frac{1}{16C^2\lambda^{\frac{\alpha}{r}}\|u_0\|_{\dot{FB}_{p,q}^{\beta}}}\Big)^r,
       \big(\frac{1}{16C^2\lambda^{\frac{\alpha}{r'}}\|u_0\|_{\dot{FB}_{p,q}^{\beta}}}\Big)^{r'}\Big\}.
    \]
\textbf{Step 3:} We have proved equation has an unique solution in $X$ under the condition:
$r'\max\{1,\sigma+1\}<\alpha<n(1-\frac{1}{p})+\sigma+2$, $1<r<\infty$.
Using the integral form (\ref{equation31}) and Lemma \ref{transport}, we can deduce
    \[
     \|u\|_{\tilde{L}^{\infty}_T(\dot{FB}^{\beta}_{p,q})}+\|u\|_{\tilde{L}^{1}_T(\dot{FB}^{\beta+\alpha}_{p,q})}
     \leq C(\|u_0\|_{\dot{FB}^{\beta}_{p,q}}+\|u\|_X).
    \]
Hence $u$ is also belongs to $C([0,T);\dot{FB}^{\beta}_{p,q})\cap{\tilde{L}^{1}_T(\dot{FB}^{\beta+\alpha}_{p,q})}$. Since
$1<r<\infty$ in $X$ can be chose to be a sufficiently large number, we in fact improve the index in Theorem \ref{theorem31}
to
    $$\max\{1,\sigma+1\}<\alpha<n(1-\frac{1}{p})+\sigma+2.$$
Besides, this improvement make no difference to the proof of Theorem 3.2.


\end{document}